\documentclass[a4paper,12pt]{amsart}
\date{May 16, 2016}

\usepackage{amssymb,amsmath}
\usepackage[dvips]{graphics}
\usepackage[dvips]{color}
\usepackage{graphicx}
\usepackage[english]{babel}
\usepackage[T1]{fontenc}       		
\usepackage{mathrsfs}   
\usepackage{comment}
\usepackage[textsize=tiny]{todonotes}
\usepackage{xcolor}

\newtheorem{Theorem}{Theorem}
\newtheorem{Lemma}[Theorem]{Lemma}
\newtheorem{Proposition}[Theorem]{Proposition}
\newtheorem{Corollary}[Theorem]{Corollary}
\newtheorem{Conjecture}[Theorem]{Conjecture}

\theoremstyle{definition}
\newtheorem{Definition}[Theorem]{Definition}
\newtheorem{Example}[Theorem]{Example}
\newtheorem{Remark}[Theorem]{Remark}

\providecommand{\ch}[1]{\text{\raise 2pt \hbox{$\chi$}\kern-0.2pt}_{#1}}

\def\XXint#1#2#3{{\setbox0=\hbox{$#1{#2#3}{\int}$}
      \vcenter{\hbox{$#2#3$}}\kern-.5\wd0}}

\newcommand{\N}{\mathbb{N}}     
\newcommand{\bbD}{\mathbb{D}}       
\newcommand{\R}{\mathbb{R}}     

\newcommand{\Z}{\mathbb{Z}}         

\newcommand{\calB}{\mathscr{B}}
\newcommand{\calC}{\mathscr{C}}
\newcommand{\calD}{\mathscr{D}}

\newcommand{\calI}{\mathscr{I}}

\newcommand{\calR}{\mathscr{R}}

\DeclareMathOperator{\diam}{diam}			

\usepackage[mathcal]{eucal}					

\renewcommand{\geq}{\geqslant}
\renewcommand{\leq}{\leqslant}
\renewcommand{\epsilon}{\varepsilon}

\usepackage{color}

\newcounter{aff}

\newenvironment{Claim}{\par \addvspace{0.2cm} \refstepcounter{aff} \noindent \textit {Claim \theaff .~}}{\par \addvspace{0.2cm}}

\newenvironment{demo}{\par \addvspace{0.2cm} \noindent\textit{Proof.~}}{\hfill$\blacksquare$ \setcounter{aff}{0} \par  \addvspace{0.2cm} }

\begin{document}

\author{Emma D'Aniello and Laurent Moonens}

\title{Averaging on $n$-dimensional rectangles}

\thanks{Laurent Moonens was partially supported by the  French ANR project ``GEOMETRYA'' no.~ANR-12-BS01-0014.}

\maketitle

\begin{abstract}In this work we investigate families of translation invariant differentiation bases $\calB$ of rectangles in $\R^n$, for which $L\log^{n-1}L(\R^n)$ is the largest Orlicz space that $\calB$ differentiates. In particular, we improve on techniques developed by A.~Stokolos in \cite{STOKOLOS1988} and \cite{STOKOLOS2008}.
\end{abstract}

\section{Introduction}
Recall that a {\it differentiation basis} in ${\Bbb R}^{n}$ is a collection ${\calB}  =\bigcup_{x \in {\Bbb R}^{n}} \calB(x)$ of bounded, measurable sets  with positive measure such that for each $x \in {\Bbb R}^{n}$ there is a subfamily ${\calB}(x)$ of sets of ${\calB}$ so that each $B \in {\calB}(x)$ contains $x$ and  in ${\calB}(x)$ there are sets  arbitrarily small diameter (see \emph{e.g.} de Guzm\'an \cite[p.~104]{Guzman1974} and \cite[p.~42]{Guzman1975}).

Let ${\calB}$ be a differential basis in ${\Bbb R}^{n}$, and let $M_{\calB}$ be the corresponding maximal functional, that is, for any {}{summable} function $f$ in $\R^n$, let for $x\in\R^n$:
$$M_{\calB} f(x)= \sup_{\begin{subarray}{c}B \in {\calB}\\B\ni x\end{subarray}}  \frac{1}{|B|} \int_{B}   \vert f\vert,$$ 
where$|A|$ denotes the Lebesgue measure of a measurable set $A\subseteq\R^n$.

In many areas of analysis a key role is played by the so-called {\it weak type estimates} for $M_{\calB} f$; this is the case, in particular, for the weak $(1,1)$ estimate:
\begin{equation} \label{equ1}
\vert \{x: M_{\calB} f(x) > \lambda \} \vert \leq C \frac{{\Vert f \Vert}_{1}}{\lambda}, \hspace{0.6cm} \lambda >0,
\end{equation}
and the weak $L\log^{d} L$ estimate for $0<d\leq n$:  
\begin{equation} \label{equ2}
\vert \{x: M_{\calB} f(x) > \lambda \} \vert \leq C \int_{{\Bbb R}^{n}} \frac{\vert f \vert}{\lambda}\left(1 + \log_{+}^{d} \frac{\vert f \vert}{\lambda}\right), \hspace{0.6cm} \lambda >0{}{,}
\end{equation}
{}{where we let $\log_+t:=\max(\log t,0)$ for $t>0$.}
Following \emph{e.g.} Stokolos \cite{STOKOLOS2005}, if $M_{\calB}f$ satisfies (\ref{equ1}) or (\ref{equ2}), we say that the basis ${\calB}$ has the corresponding weak type.  

As far as translation invariant bases $\calB$ consisting of multidimensional intervals (also called rectangles or parallelepipeds, that is Cartesian products of one dimensional intervals) are concerned, an old result by Jessen, Marcinkiewicz and Zygmund \cite[Theorem 4]{JMZ}, quantified {}{independently by Fava \cite{FAVA1972} and} de~Guzm\'an \cite{Guzman1974}, ensures that $M_\calB$ always enjoys weak type $L\log^{n-1} L${}{; this shows in particular that Zygmund's conjecture, stating that translation-invariant bases of $n$-dimensional rectangles whose sides are increasing functions of $d$ parameters differentiate $L\log^{d-1}L(\R^n)$, holds for $d=n$.}

In dimension $2$, Stokolos \cite{STOKOLOS1988} gave a complete characterization of all the weak estimates that translation invariant bases of rectangles can support: either the dyadic parents of elements of $\calB$ can be, up to translation, classified into a finite number of families totally ordered by inclusion~---~in which case $M_\calB$ satisfies a weak $(1,1)$ estimate~---~, or it is not the case and de Guzm\'an's weak $L\log L$ estimate is sharp.

In the general case of translation invariant bases of rectangles in $\R^n$, {}{Fefferman and Pipher \cite{FP2005} gave in 2005 a covering lemma yielding $L\log^d L$ estimates for some translation-invariant bases $\calB$ of rectangles in $\R^n$. T}hose providing a weak type $(1,1)$ estimate for $M_\calB$ have been completely characterized in Stokolos \cite{STOKOLOS2005}: a translation invariant base ${\calR}$ of rectangles is of weak type $(1,1)$ if and only if  the dyadic parents of elements of $\calB$ can be, up to translation, classified into a finite number of families totally ordered by inclusion (we recall this result in section~\ref{sec.3} below). {}{As far as Zygmund's conjecture is concerned, this corresponds to its validity for $d=1$.}

{}{Even though a result by C\'ordoba \cite{CORDOBA1979}, stating \emph{e.g.} that the translation-invariant basis $\calB$ in $\R^3$ verifying:
$$
\calB(0)=\{[0,s]\times [0,t]\times [0,st]:s,t>0\},
$$
yields a weak $L\log L$ estimate for $M_\calB$, hence supporting Zygmund's conjecture for $d=2$, it follows from a result by Soria \cite[Proposition~5]{SORIA1986} that Zygmund's conjecture is false in general, for he there constructs continuous, increasing functions $\phi,\psi:(0,\infty)\to (0,\infty)$ such that the translation-invariant basis of rectangles $\calB'$ in $\R^3$ satisfying:
$$
\calB'(0)=\{[0,s]\times [0,t\phi(s)]\times [0,t\psi(s)]:s,t>0\},
$$
differentiates no more than $L\log^2L(\R^3)$, meaning in particular that de Guzmán's $L\log^2L$ estimate is sharp for $M_{\calB'}$. It is also Soria's observation (see \cite[Proposition~2]{SORIA1986}) that Córdoba's result implies that the ``Soria basis'' $\calB''$ verifying:
$$
\calB''(0)=\{[0,s]\times [0,t]\times [0,1/t]:s,t>0\},
$$
yields a weak $L\log L$ inequality for $M_{\calB''}$ (note here that $\calB''$ is not, \emph{stricto sensu}, a differentiation basis~---~it lacks the ``diameter'' condition~---~, which does not prevent us to extend to such families the definitions made above).
}

As of today, there is no characterization of translation invariant bases of rectangles in $\R^n$, $n\geq 3$, for which de Guzm\'an's weak $L\log^{n-1} L$ estimate is sharp (neither is it known whether the sharpness of some weak $L\log^d L$ inequality, $d>1$, would imply that $d$ is an integer). In 2008, Stokolos \cite{STOKOLOS2008} gave examples of Soria bases in $\R^3$ (\emph{i.e.} bases of the form $\calB\times \calI$, where $\calB$ is a basis of rectangles in $\R^2$ and $\calI$ denotes the basis of all intervals in $\R$) for which de Guzm\'an's weak $L\log^2 L$ estimate is sharp.

Our intention in this work is mainly to improve on Stokolos' techniques in order to give new examples of translation invariant bases of rectangles in $\R^n$ for which the weak $L\log^{n-1} L$ estimate on $M_\calB$ is sharp (see section~\ref{sec.4} below). In particular we give a way to construct, from a (strictly) decreasing sequence of rectangles in $\R^n$, a differentiation basis built from it and failing to differentiate spaces below $L\log^{n-1}L(\R^n)$ (see Theorem~\ref{thm.stokn} and Remark~\ref{rmk.stok22} below). We finish in section~\ref{sec.5} by some observations showing that de Guzm\'an's weak $L\log^{n-1}L$ estimate can be improved once we know that, in some coordinate plane, the projections of rectangles in $\calB$ can be classified, up to translations, into a finite set of families totally ordered by inclusion.

 \section{Preliminaries}

Let us now precisely fix the context in which we shall work.
\subsection{Orlicz spaces}
Given an Orlicz function $\Phi:[0,\infty)\to [0,\infty)$ (\emph{i.e.} a convex and increasing function satisfying $\Phi(0)=0$), we define the associated Banach space $L^\Phi(\R^n)$ as the set of all measurable functions $f$ on $\R^n$ for which one has $\Phi(|f|)\in L^1$. For the Orlicz function $\Phi_d(t):=t(1+\log_+^dt)$, $0< d\leq n$, we write $L\log^d L(\R^n)$ instead of $L^{\Phi_d}(\R^n)$. Is is clear that for $\Phi(t)=t^p$, $p\geq 1$, the Orlicz space $L^\Phi(\R^n)$ coincides with the usual Lebesgue space $L^p(\R^n)$.

\subsection{Families of standard rectangles} In the sequel, $\calR$ will always stand for a family of \emph{standard rectangles}, \emph{i.e.} rectangles of the form $[0,\alpha_1]\times\cdots\times [0,\alpha_n]$ in some $\R^n$, $n\geq 2$, with $0<\alpha_i\leq 1$, $1\leq i\leq n$. We will moreover say that those are \emph{dyadic} in case one has $\alpha_i=2^{-m_i}$ with $m_i\in\N$, for each $1\leq i\leq n$ (we let as usual $\N=\{0,1,2,\dots\}$ denote the set of all natural numbers).

\subsection{Weak inequalities} Given a family $\calR$ of standard rectangles in $\R^n$, we associate to it a maximal operator $M_\calR$ defined for a measurable function $f$ by:
$$
M_\calR f(x):=\sup\left\{\frac{1}{|R|}\int_{\tau(R)} |f|: R\in\calR, \tau\text{ translation}, x\in\tau(R)\right\}.
$$
In case $\Phi$ is an Orlicz function, we shall say that $M_\calR$ satisfies a weak $L^\Phi$ inequality in case there exists $C>0$ such that, for any $\lambda>0$ and any $f\in L^\Phi(\R^n)$, we have:
$$
|\{M_\calR f>\lambda\}|\leq\int_{\R^n} \Phi\left(\frac{C|f|}{\lambda}\right).
$$

\subsection{Reduction to families of dyadic rectangles} Given a family $\calR$ of standard rectangles in $\R^n$, we define $\calR^*:=\{R^*:R\in\calR\}$, where for each standard rectangle $R$, we denote by $R^*$ the standard dyadic rectangle with the smallest measure, having the property to contain $R$. Since it is easy to see that one has:
$$
\frac{1}{2^n}\cdot M_{\calR^*} f\leq M_\calR f\leq 2^n\cdot M_{\calR^*} f,
$$
on $\R^n$ for all measurable $f$, it is obvious that $M_\calR$ satisfies a weak $L^\Phi$ inequality if and only if $M_{\calR^*}$ does. For this reason, we shall always assume in the sequel that $\calR$ is a family of standard \emph{dyadic} rectangles.

\subsection{Links to differentiation theory} Given a family of standard rectangles $\calR$ satisfying $\inf\{\diam R:R\in\calR\}=0$, one can associate to it a translation invariant differentiation basis $\calB:=\{\tau(R): R\in\calR, \tau\text{ translation}\}$ and define, for $x\in\R^n$, $\calB(x):=\{R\in\calB: R\ni x\}$. In many cases, it then follows from the Sawyer-Stein principle (see \emph{e.g.} \cite[Chapter~1]{GARSIA}) that a weak $L^\Phi$ inequality for $M_\calR$ is equivalent to having:
$$
f(x)=\lim_{\begin{subarray}{c}\diam R\to 0\\R\in\calB(x)\end{subarray}} \frac{1}{|R|} \int_{R} f\qquad\text{for a.e. }x\in\R^n,
$$
for all $f\in L^\Phi(\R^n)$ (we shall say in this case that $\calR$ \emph{differentiates} $L^\Phi(\R^n)$).
\subsection{Rademacher-type functions} {}{Recall that one denotes by $\chi_A$ the characteristic function of a set $A\subseteq\R^n$.} We define the sequence $(r_i)$ of Rademacher-type functions on $\R$ in the following way: we let $r_1=\chi_{\Z+[0,1/2)}$ and, for $i\geq 2$, we define $r_i$ by asking that, for $x\in \R$, one has $r_i(x)=r_{i-1}(2x\mod 1)$. It is easy to see that $r_i$ is $2^{1-i}$-periodic for each $i\geq 1$ and that the $r_i$'s are independent and identically distributed (IID) in the sense that we have:
$$
\int_I \prod_{l=1}^k r_{i_l} =\left(\frac 12\right)^k |I|,
$$
for any finite sequence $1\leq i_1<i_2<\cdots <i_k$ of distinct integers and any interval $I$ whose length is a multiple of $2^{-i_1}$.

\section{Comparability conditions on rectangles}\label{sec.3}
Following Stokolos \cite{STOKOLOS1988} (and using the terminology introduced in Moonens and Rosenblatt \cite{MR2012}), we say that a family of standard 
dyadic rectangles in $\R^n$ has \emph{finite width} in case it is a finite union of families of rectangles totally ordered by inclusion, and that it has 
\emph{infinite width} otherwise. Il follows from a general result by Dilworth \cite{DILWORTH} that a family of rectangles in $\R^n$ has infinite width if and 
only if it contains families of incomparable (with respect to inclusion) rectangles having arbitrary large (finite) cardinality.

The following lemma by Stokolos \cite[Lemma~1]{STOKOLOS2005} is useful to relate the weak $(1,1)$ behaviour of the maximal operator $M_\calR$ associated to a family of rectangles, and to their comparability properties.
\begin{Lemma}
A family of rectangles in $\R^n$ is a chain (with respect to inclusion) if and only if the projections of its elements on the 
$x_1x_j$ plane form a chain of rectangles in $\R^2$, for any $j=2,\dots,n$.
In particular, a family $\calR$ of rectangles in $\R^n$ has finite width if and only if, for each $j=2,\dots,n$, the family:
$$
\{p_j(R):R\in\calR\}
$$
has finite width, where $p_j:\R^n\to\R^2$ denotes the projection on the $x_1x_j$ plane.
\end{Lemma}

Using this lemma and \cite[Lemma~1]{STOKOLOS1988}, Stokolos \cite[Theorem~2]{STOKOLOS2005} obtains the 
following geometrical characterization of families of rectangles providing a weak $(1,1)$ inequality.
\begin{Theorem}[Stokolos]
Assume $\calR$ is a differentiating family of standard, dyadic rectangles in $\R^n$. The maximal operator 
$M_\calR$ satisfies a weak $(1,1)$ inequality if and only if $\calR$ has finite width.
\end{Theorem}

We intend now to discuss some examples of families of rectangles having infinite width, but providing different optimal weak type inequalities.

Our first example will deal with families of rectangles in $\R^n$ for which the de Guzm\'an's $L\log^{n-1}L$ weak inequality is optimal.

\section{Families of rectangles for which $L\log^{n-1}L$ is sharp}\label{sec.4}

In this section we provide some examples of families of rectangles $\calR$ in $\R^n$ for which de Guzm\'an's $L\log^{n-1}L$ estimate is sharp. To this purpose, we now state a sufficient condition (in the spirit of Stokolos' \cite[Lemma~1]{STOKOLOS1988}) on a family of rectangles providing this sharpness. 

\begin{Proposition}\label{prop.stokn}
Assume that $\calR$ is a family of standard dyadic rectangles in $\R^n$ and that there exists and integer $1\leq d\leq n-1$, together with positive constants $c$ and $c'\geq 1$ depending only on $n$ and $d$, having the following property: for each sufficiently large $k\in\N$, there exist sets $\Theta_k$ and $Y_k$ in $\R^n$ with the following properties:
\begin{enumerate}
\item[(i)] $\Theta_k\subseteq Y_k$;
\item[(ii)] $|Y_k|\geq  c\cdot 2^{dk} k^{d} |\Theta_k|$;
\item[(iii)] for each $x\in Y_k$, one has $M_\calR \chi_{\Theta_k}(x)\geq  c'2^{-dk}$.
\end{enumerate}
Under these assumptions, if $\Phi$ is an Orlicz function satisfying $\Phi=o(\Phi_{d})$ at $\infty$, then $M_\calR$ does \emph{not} satisfy a weak $L^\Phi$ estimate. In particular, $M_\calR$ does not satisfy a weak $(1,1)$ estimate. 
\end{Proposition}
\begin{demo}
Define, for $k$ sufficiently large, $f_k:=(1/c')\cdot 2^{dk} \chi_{\Theta_k}$, where $\Theta_k$ and $Y_k$ are associated to $k$ and $\calR$ according to (i-iii).
\begin{Claim}\label{cl.Mphi} For each sufficiently large $k$, we have:
$$
|\{M_\calR f_k\geq 1\}|\geq \kappa(n,d)\int_{\R^n} \Phi_{d}(f_k),
$$
where $\kappa(n,d):=\frac{d^d}{cc'}$ is a constant depending only on $n$ and $d$.
\end{Claim}
\begin{proof}[Proof of the claim]
To prove this claim, one observes that for $x\in Y_k$ we have $M_\calR f_k(x)\geq 1$ according to assumption (iii). Yet, on the other hand, one computes, for $k$ sufficiently large:
$$
\int_{\R^n} \Phi_{d}(f_k)\leq \frac{1}{c'}\cdot 2^{dk} |\Theta_k| [1+(dk\log 2)^{d}] \leq \frac{d^d}{c'}\cdot 2^{dk} k^{d} |\Theta_k|\leq
\kappa(n,d)  \cdot |Y_k|,
$$
and the claim follows.
\end{proof}

\begin{Claim}
For any $\Phi$ satisfying $\Phi=o(\Phi_{d})$ at $\infty$ and for each $C>0$, we have:
$$
\lim_{k\to\infty}\frac{\int_{\R^n} \Phi_{d}(|f_k|)}{\int_{\R^n} \Phi(C|f_k|)}=\infty.
$$
\end{Claim}
\begin{proof}[Proof of the claim]
Compute for any $k$:
\begin{eqnarray*}
\frac{\int_{\R^n} \Phi(C|f_k|)}{\int_{\R^n} \Phi_{d}(|f_k|)} & = & \frac{\Phi(2^{dk}C/c')}{\Phi_{d}(2^{dk}/c')}\\
& = & \frac{\Phi( 2^{dk}C/c')}{\Phi_{d}(2^{dk}C/c')}
 \frac{\Phi_{d}( 2^{dk}C/c')}{\Phi_{d}(2^{dk}/c')},
\end{eqnarray*}
observe that the quotient $ \frac{\Phi_{d}( 2^{dk}C/c')}{\Phi_{d}(2^{dk}/c')}$ is 
bounded as $k\to\infty$ by a constant independent of $k$, while by assumption the quotient 
$\frac{\Phi( 2^{dk}C/c')}{\Phi_{d}(2^{dk}C/c')}$ tends to zero as $k\to\infty$. The claim is proved.
\end{proof}

We now finish the proof of Proposition~\ref{prop.stokn}. To this purpose, fix $\Phi$ an Orlicz function satisfying $\Phi=o(\Phi_{d})$ at $\infty$ 
and assume that there exists a constant $C>0$ such that, for any $\lambda>0$, one has:
$$
|\{M_\calR f>\lambda\}|\leq\int_{\R^n} \Phi\left(\frac{C|f|}{\lambda}\right).
$$
Using Claim~\ref{cl.Mphi}, we would then get, for each $k$ sufficiently large:
$$
0<\kappa(n,d)\int_{\R^n}\Phi_{d}(f_k)\leq \left|\left\{M_\calR f_k> \frac 12 \right\}\right|\leq \int_{\R^n} \Phi({2Cf_k}),
$$ contradicting the previous claim and proving the proposition.
\end{demo}

A first situation to which the previous proposition can be applied concerns some families of rectangles containing arbitrary large subsets of rectangles having the same $n$-dimensional measure.
\subsection{A first example}
We shall need the following Lemma. 
\begin{Lemma} \label{LEMMA1}
Fix $\alpha >0$. For each $k \in \N\setminus\{0\}$, let $\bbD_k:=\{2^{-j}:0\leq j\leq k\}$ and define
$${\calR}_{k} : =  \left\{[0, s_{1}] \times \ldots \times  [0, s_{n-1}]  \times \left[0, \frac{\alpha}{s_{1} \cdot \ldots \cdot s_{n-1}}\right]: 
s_{1}, \ldots, s_{n-1} \in \bbD_k\right\}.$$
Then 
$$\vert \cup{\calR}_{k} \vert \geq \frac{1}{ 3 \cdot 2^{n-2} }{k}^{n-1} {\alpha}.$$
\end{Lemma}
\begin{demo}
Fix $k \in {\Bbb N}$. We prove this claim by induction on the dimension of the space $n$, and we observe that it follows from Moonens and 
Rosenblatt \cite[Claim~15]{MR2012} that the inequality holds true in dimension $n=2$.

Assume now that it holds in dimension $n-1$ and let us prove it in dimension $n$. To that purpose, define $E_j:=[2^{-j-1},2^{-j}]\times [0,1]^{n-1}$ 
for each $0\leq j\leq k$ and observe that one has:
$$
|\cup\calR_k|\geq\sum_{j=0}^k \left| (\cup\calR_k)\cap E_j\right|.
$$
Denoting by $s_1(R)$ the length of the first side of $R\in\calR_k$, and by $p_1':\R^n\to\R^{n-1}$ the projection on the $n-1$ last coordinates {defined by 
$p_1'(x_1,\dots,x_n)=(x_2,\dots,x_n)$, we} obviously have, for $0\leq j\leq k$:
$$
(\cup\calR_k)\cap E_j\supseteq (\cup\{R\in\calR_k:s_1(R)=2^{-j}\})\cap E_j.
$$
Since we compute, for the same $j$:
$$
|(\cup\{R\in\calR_k:s_1(R)=2^{-j}\})\cap E_j|= 2^{-j-1} |\cup\{p_1'(R):R\in\calR_k, s_1(R)=2^{-j}\}|,
$$
and since we have:
\begin{eqnarray*}
\calR'_k&:=&\cup\{p_1'(R):R\in\calR_k, s_1(R)=2^{-j}\}\\&=&\left\{ [0,s_2]\times \cdots\times 
[0,s_{n-1}]\times \left[0, \frac{2^j\alpha}{s_2\cdot \cdots s_{n-1}}\right]: s_2,\dots, s_{n-1}\in\calD_k\right\},
\end{eqnarray*}
the induction hypothesis, applied to $\calR'_k$ and $2^j$, yields:
$$
|\cup\{p_1'(R):R\in\calR_k, s_1(R)=2^{-j}\}|\geq \frac{1}{3\cdot 2^{n-3}} 2^j k^{n-2} \alpha.
$$
Now we compute:
$$
|\cup\calR_k|\geq  \frac{1}{3\cdot 2^{n-3}}  k^{n-2}\alpha \sum_{j=0}^k 2^{-j-1}\cdot 2^j\geq 
\frac{1}{3\cdot 2^{n-2}}  k^{n-1} \alpha;
$$
the proof is complete.
\end{demo}

Using the previous estimate, we can show the following theorem.
\begin{Theorem}\label{thm.stokn2}
Let $\calR:=\bigcup_{k\in\N}\calR_k$, where for each $k\in\N$, $\calR_k$ is defined as in the statement of Lemma~\ref{LEMMA1} with $\alpha$ replaced by $\alpha_k:=2^{-nk}$. Under those conditions, de Guzm\'an's $L\log^{n-1}L$ weak estimate for the maximal operator $M_\calR$ is optimal in the following sense: if $\Phi$ is an Orlicz function satisfying $\Phi=o(\Phi_{n-1})$ at $\infty$, then $M_\calR$ does \emph{not} satisfy a weak $L^\Phi$ estimate.
\end{Theorem}

It is clear, according to Proposition~\ref{prop.stokn}, that in order to prove Theorem~\ref{thm.stokn2} we simply need to show the following lemma.
\begin{Lemma}\label{lem.stokn2}
Let $\calR:=\bigcup_{k\in\N}\calR_k$, where for each $k\in\N$, $\calR_k$ is defined as in the statement of Lemma~\ref{LEMMA1} with $\alpha$ replaced by $\alpha_k:=2^{-nk}$. Under those assumptions, for each $k\in\N$, there exist sets $\Theta_k\subseteq [0,1]^n$ and $ Y_k\subseteq [0,1]^n$ satisfying the following properties:
\begin{enumerate}
\item[(i)] $\Theta_k\subseteq Y_k$;
\item[(ii)] $|Y_k|\geq  \frac{1}{3\cdot 2^{n-2}}\cdot 2^{(n-1)k} k^{n-1} |\Theta_k|$;
\item[(iii)] for each $x\in Y_k$, one has $M_\calR \chi_{\Theta_k}(x)\geq  2^{(1-n)k}$.
\end{enumerate}
\end{Lemma}
\begin{demo}
To prove this lemma, fix $k\in\N$ and let $\Theta_k:=\cap\calR_k$ and $Y_k:=\cup\calR_k$, so that (i) is obvious. One also observes that $\Theta_k$ is a rectangle whose measure satisfies $|\Theta_k|=2^{(1-n)k}\alpha_k=2^{(1-2n)k}$.
According to Lemma~\ref{LEMMA1}, we get:
$$
|Y_k|\geq\frac{1}{3\cdot 2^{n-2}} k^{n-1} 2^{-nk}=\frac{1}{3\cdot 2^{n-2}}  2^{(n-1)k} k^{n-1} \cdot 2^{(1-2n)k}=\frac{1}{3\cdot 2^{n-2}}\cdot 2^{(n-1)k} k^{n-1} |\Theta_k|,
$$
which completes the proof of (ii).\\

To show (iii), observe that, for $x\in Y_k$, there exists $R\in\calR_k$ such that one has $x\in R$; since we have $\Theta_k\subseteq R$, this yields:
$$
M_\calR\chi_{\Theta_k}(x)\geq \frac{|\Theta_k\cap R|}{|R|}=\frac{|\Theta_k|}{|R|}=2^{(1-n)k},
$$
and so (iii) is proved.
\end{demo}

Another interesting situation can be dealt with according to Proposition~\ref{prop.stokn}.
\subsection{Another series of examples after Stokolos' study of Soria bases in $\R^3$ \cite{STOKOLOS2008}}

In this section, we write $R\prec R'$ for two standard dyadic rectangles $R=\prod_{i=1}^n [0,a_i]$ and $R'=\prod_{i=1}^n [0,b_i]$ in case one has $a_i<b_i$ for all $1\leq i\leq n$. A \emph{strict chain} of dyadic rectangles is then a finite family  of standard dyadic rectangles, for which one has $R_0\prec R_1\prec\cdots\prec R_k$.

Given $\calC=\{R_0,\dots,R_k\}$ ($k\in\N$) a strict chain of standard dyadic rectangles (say that one has $R_0\prec R_1\prec\cdots\prec R_k$), we denote by $\widehat{\calC}$ the family of standard dyadic rectangles obtained as intersections $\bigcap_{i=1}^n R^i_{j_i}$, where $k\geq j_1\geq\cdots\geq j_n\geq 0$ is a nonincreasing sequence of integers, and where we let, for $1\leq i\leq n $ and $0\leq j\leq k$:
$$
R^i_j:=\begin{cases}R_j&\text{if }i=1,\\p_{1,\dots, i-1}(R_k)\times p_{i,i+1,\dots, n}(R_j)&\text{if }i\geq 2,\end{cases}
$$
with $p_{1,\dots, i-1}$ (resp. $p_{i,i+1,\dots n}$) denoting the projection on the $i-1$ first (resp. $n-i+1$ last) coordinates.

{}{On Figure~\ref{fig.1}, we represent a strict chain of rectangles $\calC$ and the 6 rectangles belonging to the associated family $\widehat{\calC}$; for the sake of clarity, we give a picture of each element of $\widehat{\calC}$ in Figure~\ref{fig.2}.}

\begin{figure}[h]
\includegraphics[width=6cm]{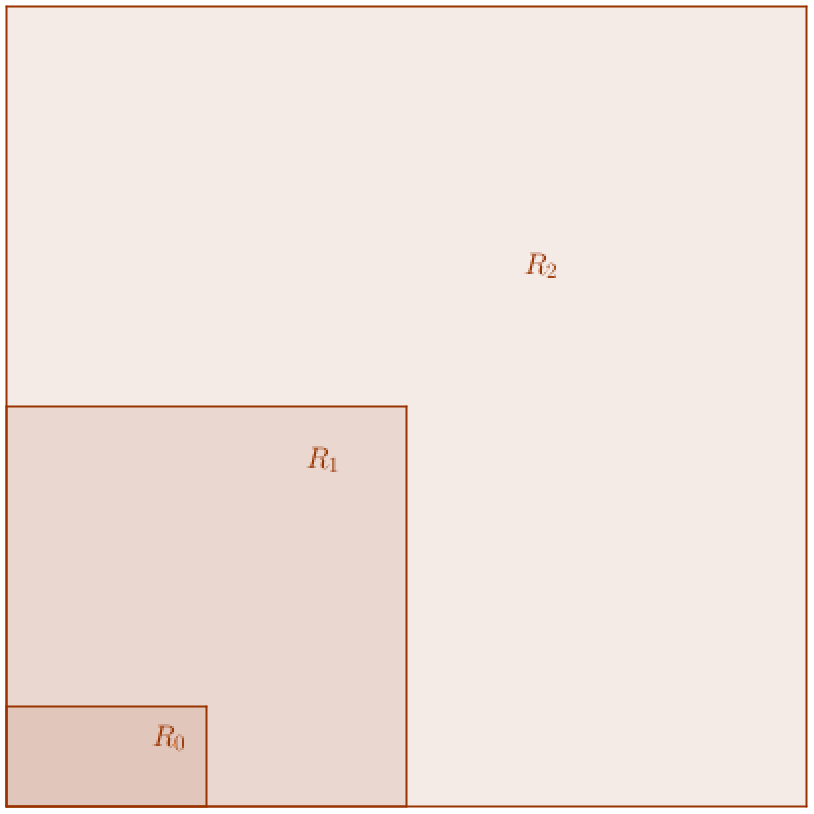}~~~~\includegraphics[width=6cm]{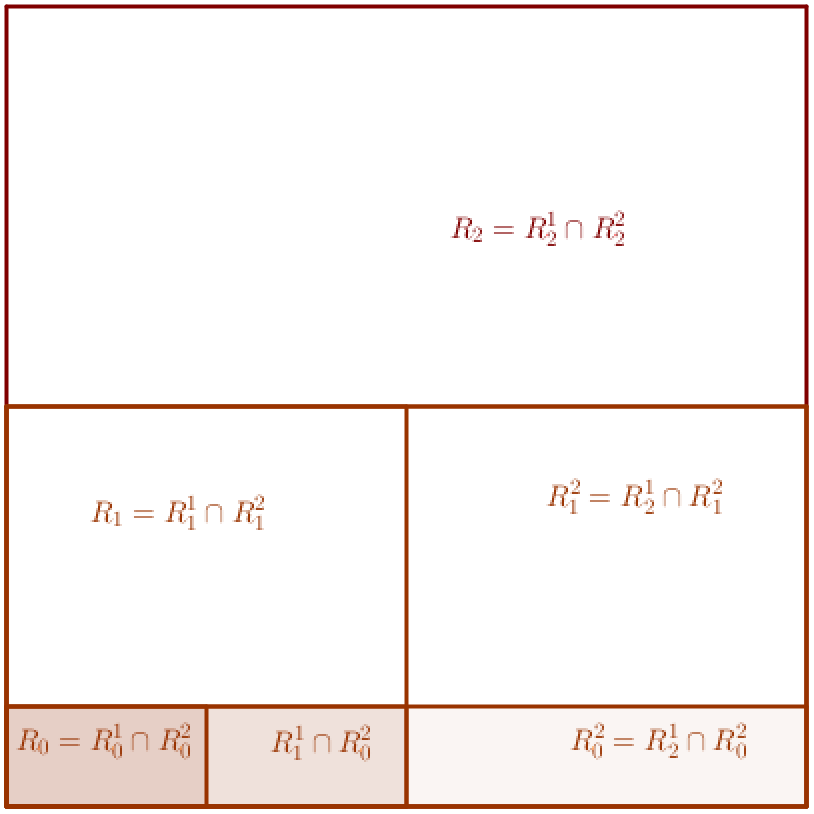}\\
\caption{A strict chain of dyadic rectangles $\calC=\{R_0,R_1,R_2\}$ (left) and the associated family $\widehat{\calC}$ (right).}\label{fig.1}
\end{figure}

\begin{figure}[h]
\includegraphics[width=6cm]{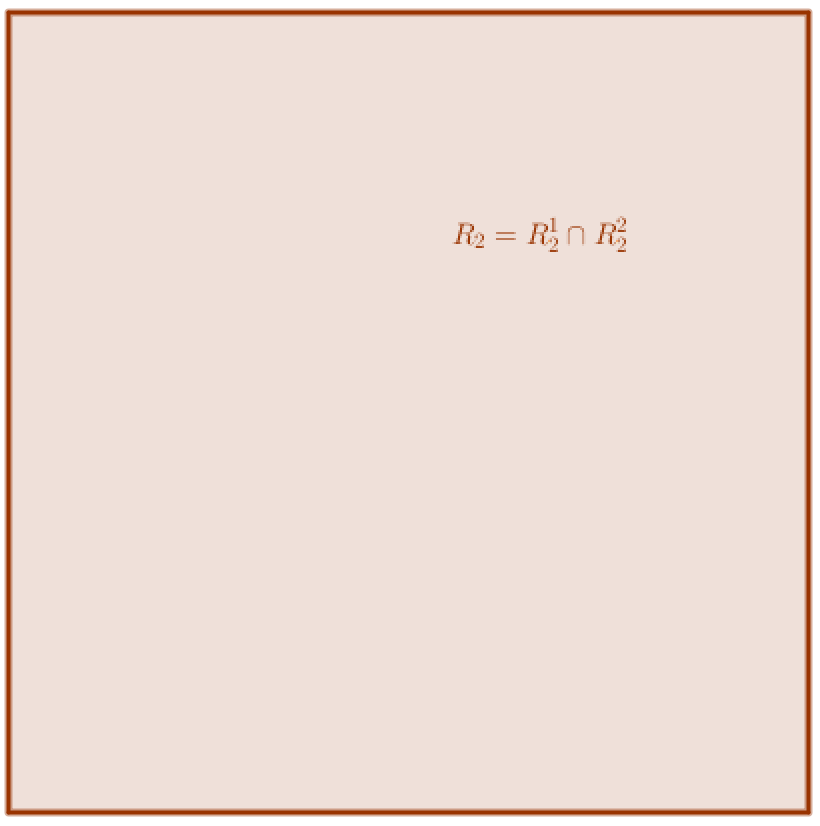}~~~~\includegraphics[width=6cm]{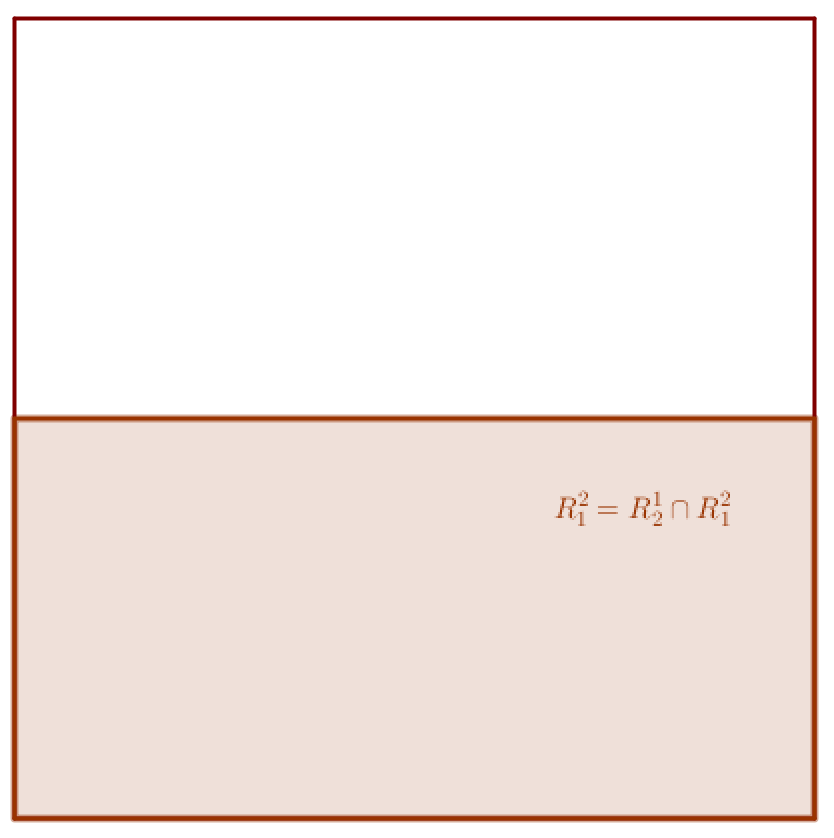}\\
 \includegraphics[width=6cm]{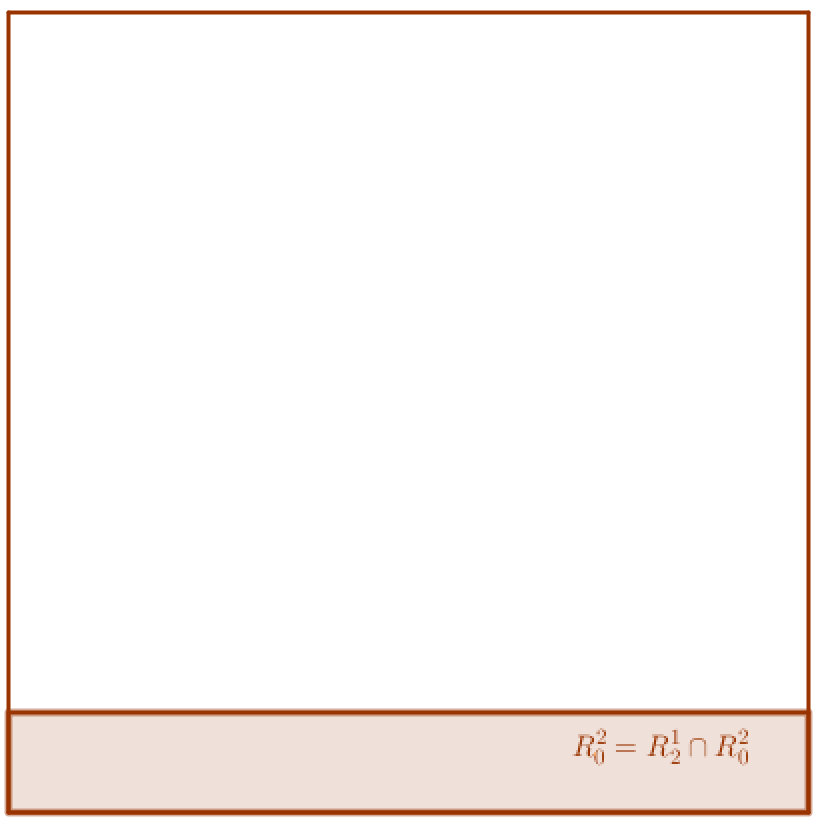}~~~~\includegraphics[width=6cm]{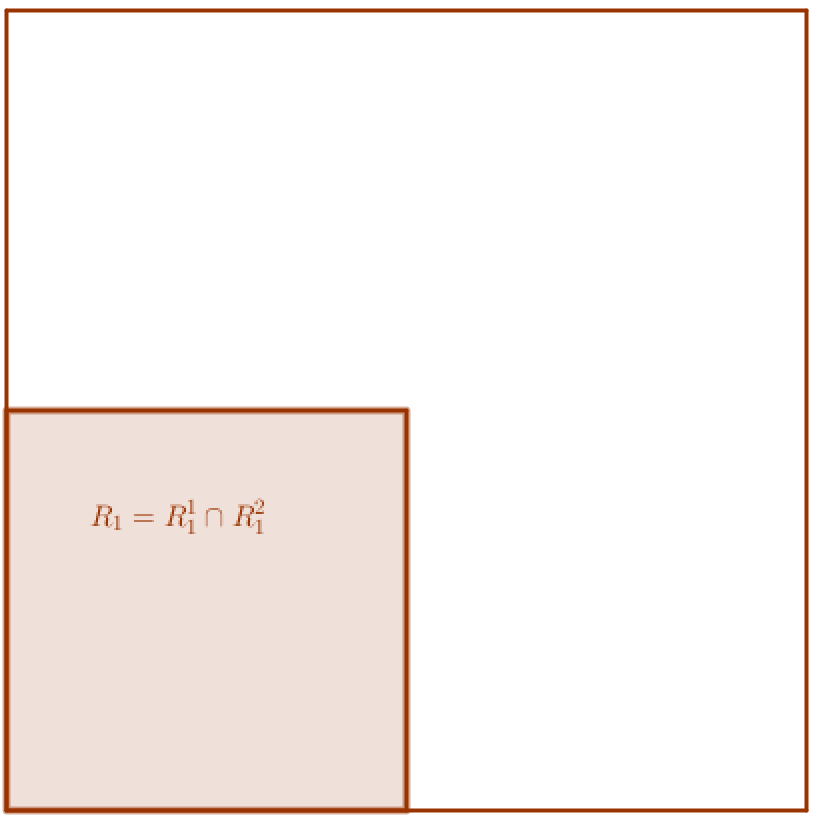}\\
\includegraphics[width=6cm]{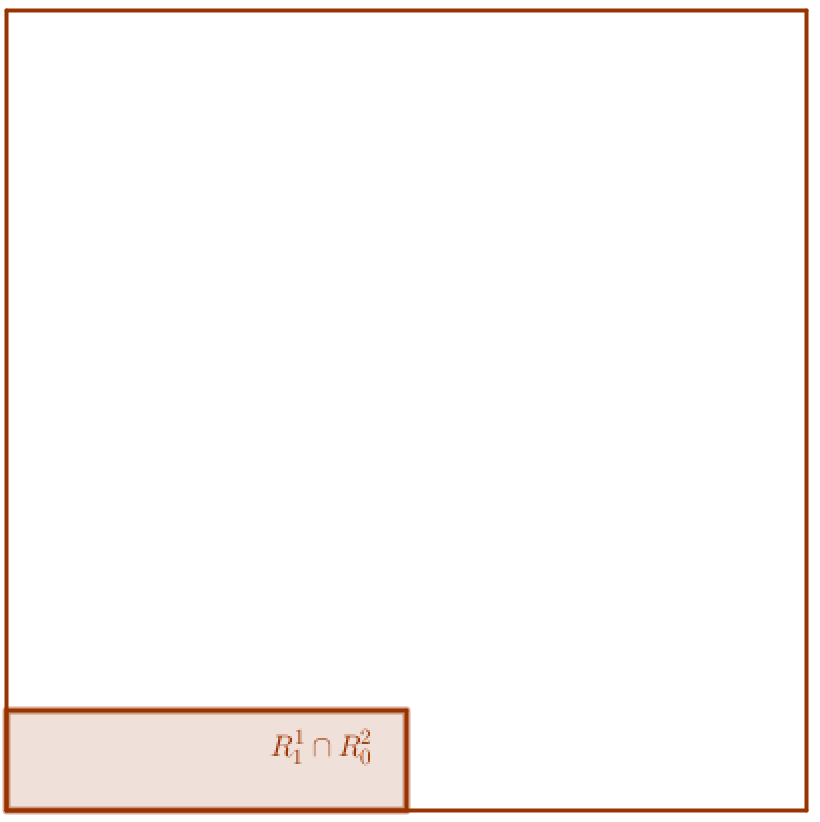}~~~~\includegraphics[width=6cm]{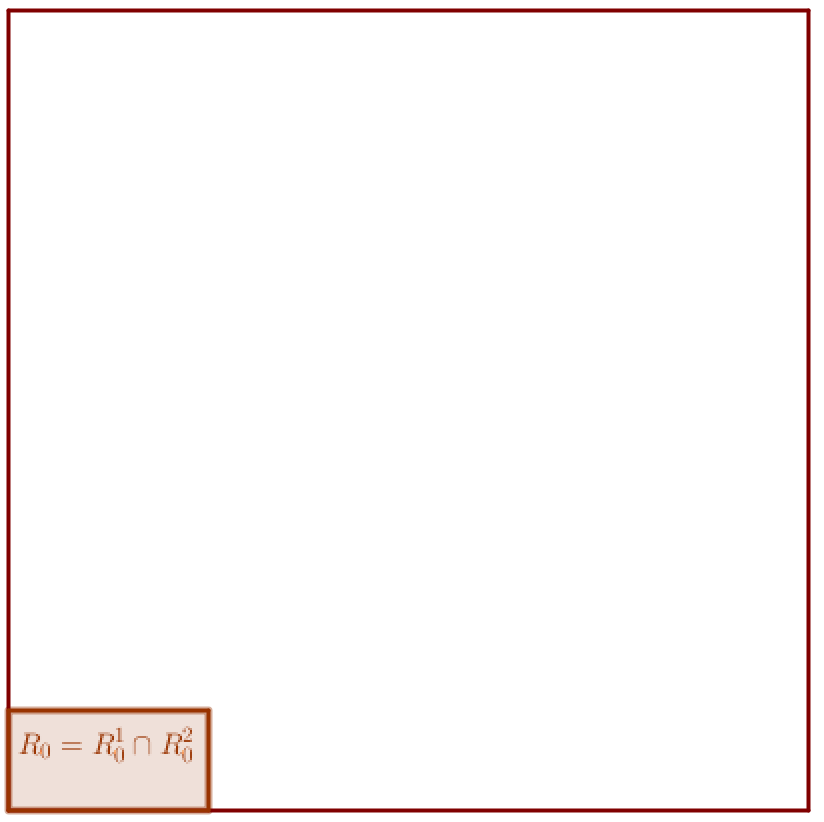}\\
\caption{The rectangles belonging to $\widehat{\calC}$ for $\calC=\{R_0,R_1,R_2\}$ as in Figure~\ref{fig.1}.}\label{fig.2}
\end{figure}
\newpage

\begin{Theorem}\label{thm.stokn}
Assume that $\calR$ is a family of standard, dyadic rectangles in $\R^n$ having the following property: for all $k\in\N$, there exist a dyadic number $\alpha\leq 2^{-k-1}$ and a strict chain $\calC$ of standard dyadic rectangles in $\R^{n-1}$ satisfying $\#\calC=k+1$ and such that one has:
$$
\calR\supseteq \left\{R\times [0,\alpha]: R\in\widehat{\calC}\right\}.
$$
Then de Guzm\'an's $L\log^{n-1}L$ weak estimate for the maximal operator $M_\calR$ is optimal in the 
following sense: if $\Phi$ is an Orlicz function satisfying $\Phi=o(\Phi_{n-1})$ at $\infty$, then $M_\calR$ does \emph{not} satisfy a weak $L^\Phi$ estimate.
\end{Theorem}
\begin{Remark}\label{rmk.stok2008}
The preceeding theorem includes as a particular case Stokolos' result \cite{STOKOLOS2008} about Soria bases in $\R^3$. 
More precisely, Stokolos there shows an estimate similar to the one we get in Claim~\ref{cl.Mphi} above with $l=n-1$, for families of standard dyadic rectangles in $\R^3$ of the form:
$$
\calR=\{R\times I: R\in\calR', I\in\calI_d\},
$$
where $\calI_d$ is the family of all dyadic intervals, and where $\calR'$ is a family of standard dyadic rectangles in $\R^2$ containing arbitrary 
large (in the sense of cardinality) finite subsets $\calR''$ having the following property:
\begin{itemize}
\item[(is)] for all $R,R'\in\calR''$ with $R\neq R'$, we have $R\nsim R'$ and $R\cap R'\in\calR'$;
\end{itemize}
here we mean by $R\nsim R'$ that neither $R\subseteq R'$ nor $R'\subseteq R$ holds.

We now show that our Theorem~\ref{thm.stokn} applies to this situation. For a given $k\in\N$, we denote by $\calR'_k$ a subfamily of 
$\calR'$ satisfying $\# \calR_k'=2k+1$ as well as property (is), and we write $\calR'_k=\{[0,\alpha_j]\times [0,\beta_j]: 0\leq j\leq 2k\}$ 
with $\alpha_0<\alpha_1<\cdots<\alpha_{2k}$. Then, we have $\beta_0>\beta_1>\cdots >\beta_{2k}$ (using the pairwise incomparability of elements in $\calR_k'$). 
Now let $\alpha^1_j:=\alpha_j$ and $\alpha^2_j:=\beta_{2k-j}$ for $0\leq j\leq k$, define for $0\leq j\leq k$:
$$
R_j:=[0,\alpha_j^1]\times [0,\alpha_j^2],
$$
and observe that one has $R_0\prec R_1\prec\cdots\prec R_k$. Letting $\calC:=\{R_0,\dots,R_k\}$ we see that for $k\geq j_1\geq j_2\geq 0$ we have, using property (is):
\begin{multline*}
R^1_{j_1}\cap R^2_{j_2}=R_{j_1}\cap [p_{1}(R_{k})\times p_2(R_{j_2})]=([0,\alpha^1_{j_1}]\times [0,\alpha^2_{j_1}])\cap ([0,\alpha^1_k]\times [0,\alpha^2_{j_2}])\\=
[0,\alpha_{j_1}^1]\times [0,\alpha^2_{j_2}]=
([0, \alpha_{j_1}]\times [0,\beta_{j_1}])\cap
([0, \alpha_{2k-j_2}]\times [0, \beta_{2k-j_2}])\in\calR'.
\end{multline*}
Hence we get $\calR\supseteq \{R\times [0,2^{-k-1}]:R\in\widehat{\calC}\}$,
and it is now clear that the hypotheses of Theorem~\ref{thm.stokn} are satisfied.
\end{Remark}

\begin{Remark}\label{rmk.stok3}
Assume that $k\in\N$ is an integer and that $\calC=\{R_0,\dots, R_k\}$ is a strict chain of standard dyadic rectangles in $\R^n$ with $R_0\prec R_1\prec\cdots\prec R_k$. Write, for each $0\leq j\leq k$, $R_j:=\prod_{i=1}^n [0,\alpha^i_j]$. Observe now that, for $k\geq j_1\geq\cdots\geq j_n\geq 0$, one computes:
$$
\prod_{i=1}^n [0,\alpha^i_{j_i}] =\bigcap_{i=1}^n R^i_{j_i}\in \widehat{\calC}.
$$
\end{Remark}

In order to prove Theorem~\ref{thm.stokn}, it is sufficient, according to Proposition~\ref{prop.stokn}, to prove the following lemma, improving on Stokolos techniques in \cite{STOKOLOS1988} and \cite{STOKOLOS2008}, and 
using Radema\-cher functions as in \cite{MR2012}.
\begin{Lemma}\label{lem.stokn}
Assume the family $\calR$ satisfies the hypotheses of Theorem~\ref{thm.stokn}. Then, for each $k\geq 2n-6$, there exist
sets $\Theta_k$ and $Y_k$ satisfying the following conditions:
\begin{enumerate}
\item[(i)] $\Theta_k  \subseteq Y_k$;
\item[(ii)] $|Y_k|\geq  \frac{2^{5-3n}}{(n-1)!} \cdot 2^{(n-1)k} k^{n-1} |\Theta_k|$;
\item[(iii)] for each $x\in Y_k$, one has $M_\calR \chi_{\Theta} (x)\geq 2^{n-1}\cdot 2^{(1-n)k}$.
\end{enumerate}
\end{Lemma}
\begin{demo}
According to Remark~\ref{rmk.stok3}, the hypothesis made on $\calR$ implies that, for each $k\in\N$, one can find increasing sequences of dyadic numbers $\alpha_0^i<\alpha_1^i<\cdots <\alpha_k^i$, $1\leq i\leq n-1$ and an integer $p\geq k+1$ such that, for any nonincreasing sequence $k\geq j_1\geq j_2\geq\cdots\geq j_{n-1}\geq 0$, one has:
\begin{equation}\label{eq.stok-vs-dm}
\left(\prod_{i=1}^{n-1} [0,\alpha^i_{j_i}]\right)\times [0, 2^{-p}]\in\calR.
\end{equation}
Given $0\leq j\leq k$, let $\alpha^i_j= 2^{-m^i_j}$ for $1\leq i\leq n$ and define $m_j^n:=p-j\geq 1$. Define also $R_0:=\prod_{i=1}^{n} [0,2^{-m^i_k}]$. 
Denote by $C(k + 1, n-1)$ the set of all nonincreasing $(n-1)$-tuples $J=(j_1,\dots, j_{n-1})$ of integers in $\{0,1,\dots, k\}$ and note that one has:
$$
\# C(k+1,n-1) \geq C_{k+1}^{n-1}=\frac{(k+1) k \cdots (k-n+3)}{(n-1)!}\geq \frac{1}{2^{n-3} (n-1)!} k^{n-1},
$$
given that $k\geq 2n-6$.

Define a set $\Theta\subseteq\R^n$ (to avoid unnecessary indices here, we write $\Theta$ and $Y$ instead of $\Theta_k$ and $Y_k$, since $k$ 
remains unchanged in this whole proof) by asking that, for any $x\in\R^n$, one has:
$$
\chi_{\Theta}(x)=\prod_{i=1}^n \prod_{j_i=1}^k r_{m^i_{j_i}}(x_i).
$$
For $J=(j_1,\dots, j_{n-1})\in C(k+1,n-1)$, let $j_0:=k$, $j_{n}:=0$ and define a set $Y_J\subseteq\R^n$ by asking that, for any $x\in\R^n$, one has:
$$
\chi_{Y_J} (x)=\prod_{i=1}^n \prod_{\mu_i=j_i}^{j_{i-1}} r_{m^i_{\mu_i}}(x_i),
$$
and let $Y:=\bigcup_{J\in C(k+1,n-1)} Y_J$. Clearly, (i) holds.

It is clear, since the Rademacher functions $(r_i)$ form an IID sequence, that one 
has $|\Theta|=2^{-nk}|R_0|$ and, for $J\in C(k +1,n-1)$, $|Y_J|=2^{1-k-n}|R_0|$.

We now write, for $1\leq i\leq n-1$
\begin{equation}\label{eq.Y}
Y=\bigcup_{j_1=0}^k \bigcup_{j_2=0}^{j_1}\cdots \bigcup_{j_{i-1}=0}^{j_{i-2}} \bigcup_{J'\in C(j_{i-1} +1,n-i)} Y_{j_1,\dots j_{i-1}, J'}.
\end{equation}
Hence, we define, for $1\leq i\leq n$ and $k\geq j_1\geq j_2\geq\cdots \geq j_{i-1}\geq 0$:
$$
E_{j_1,\dots, j_{i-1}}:=\bigcup_{J'\in C(j_{i-1} +1,n-i)} Y_{j_1,\dots j_{i-1}, J'};
$$
with this definition, we get in particular $E_{j_1,j_2,\dots, j_{n-1}}=Y_{j_1,\dots, j_{n-1}}$ for $k\geq j_1\geq j_2\geq\cdots\geq j_{n-1}\geq 0$.

\begin{Claim}\label{cl.E}
For all $1\leq r\leq n$ and $k\geq j_1\geq j_2\geq\cdots \geq j_{r-2}\geq 0$, we have:
$$
\left|\bigcup_{j=0}^{j_{r-2}} E_{j_1,\dots, j_{r-2}, j}\right|\geq\frac 12 \sum_{j=0}^{j_{r-2}} |E_{j_1,\dots, j_{r-2}, j}|.
$$
\end{Claim}
\begin{proof}[Proof of the claim]
To prove this claim, we write:
$$
\left|\bigcup_{j=0}^{j_{r-2}} E_{j_1,\dots, j_{r-2}, j}\right|=\sum_{j=0}^{j_{r-2}} \left[ |E_{j_1,\dots, j_{r-2}, j}|-\left| E_{j_1,\dots, j_{r-2}, j}
\cap\bigcup_{l=0}^{j-1} E_{j_1,\dots, j_{r-2}, l}\right|\right].
$$
Assuming that $0\leq j\leq j_{r-2}$ and $x\in E_{j_1,\dots, j_{r-2}, j}\cap\bigcup_{l=0}^{j-1} E_{j_1,\dots, j_{r-2}, l}$ are given, we find both:
\begin{multline*}
\prod_{i=1}^{r-2} \left[\prod_{\mu_i=j_i}^{j_{i-1}} r_{m^i_{\mu_i}}(x_i) \right]\cdot\prod_{\nu=j}^{j_{r-2}} r_{m^{r-1}_{\nu}}(x_{r-1})
 \\\cdot \max_{J'\in C(j +1,n-r)} \left[\prod_{\xi=j_{r+1}}^{j} r_{m^{r}_{\xi}}(x_r)\prod_{i=r+1}^n \prod_{\mu_i=j_i}^{j_{i-1}} r_{m^i_{\mu_i}}(x_i)\right]
=1,
\end{multline*}
and, for some $1\leq l\leq j-1$:
\begin{multline*}
\prod_{i=1}^{r-2} \left[\prod_{\mu_i=j_i}^{j_{i-1}} r_{m^i_{\mu_i}}(x_i) \right]\cdot\prod_{\nu=l}^{j_{r-2}} r_{m^{r-1}_{\nu}}(x_{r-1})
\\\cdot \max_{J'\in C(l +1,n-r)} \left[\prod_{\xi=j_{r+1}}^{l} r_{m^{r}_{\xi}}(x_r)\prod_{i=r+1}^n \prod_{\mu_i=j_i}^{j_{i-1}} r_{m^i_{\mu_i}}(x_i)\right]= 1,
\end{multline*}
so that we have:
\begin{eqnarray}\nonumber \lefteqn{ r_{m^{r-1}_{j-1}}(x_{r-1})\chi_{E_{j_1,\dots, j_{r-2}, j}}(x)}&&\\\label{eq.sttt}
& = & \prod_{i=1}^{r-2} \left[\prod_{\mu_i=j_i}^{j_{i-1}} r_{m^i_{\mu_i}}(x_i) \right]\cdot\prod_{\nu=j-1}^{j_{r-2}} r_{m^{r-1}_{\nu}}(x_{r-1})\\\nonumber&&
\cdot \max_{J'\in C(j +1,n-r)} \left[\prod_{\xi=j_{r+1}}^{j} r_{m^{r}_{\xi}}(x_r)\prod_{i=r+1}^n \prod_{\mu_i=j_i}^{j_{i-1}} r_{m^i_{\mu_i}}(x_i)\right]\\
\nonumber & = & 1.
\end{eqnarray}
We then get:
\begin{eqnarray*}
&  & \left| E_{j_1,\dots, j_{r-2}, j}\cap\bigcup_{l=0}^{j-1} E_{j_1,\dots, j_{r-2}, l}\right| \leq \int_{R_0}r_{m^{r-1}_{j-1}}(x_{r-1})\chi_{E_{j_1,\dots, j_{r-2}, j}}(x)\,dx\\
& = & \frac 12 \int_{R_0}\chi_{E_{j_1,\dots, j_{r-2}, j}}\\
& = & \frac 12 |E_{j_1,\dots, j_{r-2}, j}|,\\
\end{eqnarray*}
using the IID property of the Rademacher functions, Fubini's theorem and the fact that the first series 
of factors in (\ref{eq.sttt}) only depend on $x_1,\dots, x_{r-1}$. The proof of the claim is hence complete.
\end{proof}

We now turn to the proof of (ii). Observe that it now follows from (\ref{eq.Y}) and from Claim~\ref{cl.E}
that we have:
$$
|Y|\geq \frac{1}{2^{n-1}}\sum_{J\in C(k +1,n-1)} |Y_J|,
$$
for we recall that given $J=(j_1,\dots, j_{n-1})\in C(k+1,n-1)$, one has $E_{j_1,\dots, j_{n-1}}=Y_J$. This yields:
$$
|Y|\geq 2^{2-2n-k} |R_0| \# C(k+1,n-1)\geq \frac{2^{5-3n-k}}{(n-1)!} k^{n-1} |R_0|= \frac{2^{5-3n}}{(n-1)!} \cdot 2^{(n-1)k}k^{n-1} |\Theta|,
$$
which completes the proof of (ii).

To prove (iii), fix $x\in Y$ and choose $J\in C(k+1,n-1)$ such that one has $x\in Y_J$. Now observe that $Y_J$ is the disjoint union of 
rectangles, the lenghts of whose $n$ sides are $2^{-m^1_{j_1}}$,  $2^{-m^2_{j_2}}$,  $\cdots$,  $2^{-m^{n-1}_{j_{n-1}}}$ and $2^{-p}$; 
letting 
$R_J:=[0,2^{-m^1_{j_1}}]\times [0,2^{-m^2_{j_2}}]\times\cdots\times [0,2^{-m^{n-1}_{j_{n-1}}}]\times [0,2^{-p}]$, we see that $R_J\in \calR$ and that there exists a translation $\tau$ of $\R^n$ for which we have $x\in\tau(R_j)$. Periodicity conditions moreover show that we have:
$$
\frac{|\tau(R_J)\cap \Theta|}{|R_J|}=\frac{|R_J\cap \Theta|}{|R_J|}=\frac{|Y_J\cap\Theta|}{|Y_J|}=\frac{|\Theta|}{|Y_J|}= 2^{n-1}\cdot 2^{(1-n)k},
$$
so that we can write:
$$
M_\calR \chi_{\Theta}(x)\geq\frac{1}{|R_J|} \int_{\tau(R_J)} \chi_\Theta=\frac{|\tau(R_J)\cap\Theta|}{|R_J|}\geq 2^{n-1}\cdot 2^{(1-n)k},
$$
which finishes to prove (iii), and completes the proof of the lemma.
\end{demo}

\begin{Remark}\label{rmk.stok22}
A careful look as Stokolos' proof of \cite[Theorem, p.~491]{STOKOLOS2008} on Soria bases (see Remark~\ref{rmk.stok2008} above) shows that, in case $n=3$, it is a condition of type (\ref{eq.stok-vs-dm}) that Stokolos actually uses in his proof to get an inequality similar to the one obtained in Claim~\ref{cl.Mphi} for the maximal operator $M_\calR$ when $\calR$ is a Soria basis.

It is clear also from the proof of Lemma~\ref{lem.stokn} that condition (\ref{eq.stok-vs-dm}) is the one we really use to prove Theorem~\ref{thm.stokn}. Formulated as in Theorem~\ref{thm.stokn} though, this condition reads as a possibility of finding a sequence of dyadic numbers $(\alpha_k)$ tending to zero, such that the family obtained by projecting, on the $n-1$ first coordinates, the rectangles in $\calR$ having their last side equal to $\alpha_k$, contains a large number of rectangles, many of them being incomparable, and which enjoys some intersection property.

Theorem~\ref{thm.stokn} also gives a way of constructing, from a strict chain of standard, dyadic rectangles, a translation-invariant basis $\calB$ in $\R^n$ for which $L\log^{n-1} L(\R^n)$ is the optimal Orlicz space to be differentiated by $\calB$.
\end{Remark}

We now turn to examine how the $L\log^{n-1} L$ weak estimate can be improved with some additional information on the \emph{projections on a coordinate plane} of the rectangles belonging to the family under study.

\section{Some conditions on projections}\label{sec.5}

When a family of rectangles projects on a coordinate plane onto a family of $2$-dimensional rectangles having finite width, a weak $L\log^{n-2}L$ inequality holds.
\begin{Proposition} \label{Prop1}
Fix $\calR$ a family of standard dyadic rectangles in $\R^n$, $n\geq 2$, and assume that there exists a projection $p:\R^n\to\R^2$ onto one of the coordinate planes, 
such that the family $\calR':=\{p(R):R\in\calR\}$ has finite width. Under these assumptions, there exists a constant $c>0$ such that, 
for any measurable $f$ and any $\lambda>0$, one has:
$$
|\{M_\calR f>\lambda\}|\leq c\int_{\R^n} \frac{|f|}{\lambda}\left(1+\log_+^{n-2}\frac{|f|}{\lambda}\right).
$$
\end{Proposition}
\begin{Remark}
When $n=3$, this result is due to Stokolos (see \cite{STOKOLOS2005}). We mention it here in dimension $n$ since the proof is straightforward and yields an interesting comparison with the examples in section~\ref{sec.4} above.
\end{Remark}
\begin{demo}
We can assume without loss of generality that $p:\R^n\to\R^2, (x_1,\dots, x_n)\mapsto (x_{n-1},x_n)$ is the projection onto the $x_{n-1}x_n$-plane. Let us prove the result by recurrence on the dimension $n$. It is clear that the result holds if $n=2$, for then $\calR=\calR'$ and hence, it is standard (see \emph{e.g.} \cite[Theorem~10]{MR2012}) 
to see that the maximal operator $M_{\calR}$ satisfies a weak $(1,1)$ inequality.

So fix now $n\geq 3$ and assume that the results holds in dimension $n-1$. Denote by 
$p_1:\R^n\to\R, (x_1\dots, x_n)\mapsto x_1$ the projection on the first axis and by 
$p_1':\R^n\to\R^{n-1}, (x_1,\dots, x_n)\mapsto (x_2,\dots, x_n)$ the projection on the last $n-1$ coordinates, and observe that, by hypothesis, the families:
$$
\calR_1:=\{p_1(R):R\in\calR\}\quad\text{and}\quad \calR_2:=\{p_1'(R):R\in\calR\}
$$
satisfy weak inequalities as in de Guzm\'an \cite[p.~186]{Guzman1974} with $\varphi_1(t):=t$ and $\varphi_2(t):=t(1+\log_+^{n-3}t)$. 
According to \cite[Theorem, p. 50]{Guzman1975} and to the obvious inclusion $\calR_1\times\calR_2\supseteq \calR$,
we then have, for each measurable $f$ and each $\lambda>0$:
$$
|\{M_\calR f>\lambda\}|\leq \varphi_2(1)\int_{\R^n} \varphi_1\left(\frac{2|f|}{\lambda}\right)+\int_{\R^n}\left[
\int_{1}^{\frac{4|f(x)|}{\lambda}} \varphi_1\left(\frac{4|f(x)|}{\lambda\sigma}\right)\,d\varphi_2(\sigma)\right]\,dx.
$$
Now compute, for $t>1$ and $n\geq 4$:
$$
\varphi_2'(t)=1+\log^{n-3} t+(n-3)\log^{n-4} t\leq 1+(n-2)\log^{n-3} t,
$$
inequality which also holds for $n=3$. We hence have, for $x\in\R^n$ and $\lambda>0$:
$$
\int_{1}^{\frac{4|f(x)|}{\lambda}} \varphi_1\left(\frac{4|f(x)|}{\lambda\sigma}\right)\,d\varphi_2(\sigma)\leq
\int_{\R^n} \frac{4|f(x)|}{\lambda} \left[\int_1^{\frac{4|f(x)|}{\sigma}} \frac{1+(n-2)\log^{n-3} \sigma}{\sigma}\,d\sigma\right]\,dx,
$$
yet we compute:
$$
\int_1^{\frac{4|f(x)|}{\sigma}} \frac{1+(n-2)\log^{n-3} \sigma}{\sigma}\,d\sigma\leq \log_+\frac{4|f(x)|}{\lambda}+\log_+^{n-2}\frac{4|f(x)|}{\lambda},
$$
so that we can write:
$$
|\{M_\calR f>\lambda\}|\leq 2\int_{\R^n} \frac{|f|}{\lambda}\left[ 1+ 2\log_+\frac{4|f|}{\lambda}+2\log_+^{n-2} \frac{4|f|}{\lambda}\right].
$$
Yet we have on one hand:
$$
\int_{\{4|f|\leq \lambda e\}} \frac{|f|}{\lambda}\left[ 1+ 2\log_+\frac{4|f|}{\lambda}+2\log_+^{n-2} 
\frac{4|f|}{\lambda}\right]\leq 5\int_{\{4|f|\leq\lambda e\}} \frac{|f|}{\lambda},
$$
and on the other hand:
$$
\int_{\{4|f|> \lambda e\}} \frac{|f|}{\lambda}\left[ 1+ 2\log_+\frac{4|f|}{\lambda}+2\log_+^{n-2} 
\frac{4|f|}{\lambda}\right]\leq \int_{\{4|f|>\lambda e\}} \frac{|f|}{\lambda}\left[1+4\log_+^{n-2}\frac{4|f|}{\lambda}\right].
$$
Summing up these inequalities we obtain:
$$
|\{M_\calR f>\lambda\}|\leq 10\int_{\R^n} \frac{|f|}{\lambda}+
8\log^{n-2} 4\int_{\R^n} \frac{|f|}{\lambda}\left(1+\log_+\frac{|f|}{\lambda}\right)^{n-2}.
$$
Computing again:
$$
\int_{\{|f|\leq\lambda e\}} \frac{|f|}{\lambda}\left(1+\log_+\frac{|f|}{\lambda}\right)^{n-2}\leq 2^{n-2} \int_{\R^n} \frac{|f|}{\lambda},
$$
as well as:
$$
\int_{\{|f|>\lambda e\}} \frac{|f|}{\lambda}\left(1+\log_+\frac{|f|}{\lambda}\right)^{n-2}\leq 2^{n-2} \int_{\R^n} \frac{|f|}{\lambda} \log_+^{n-2}\frac{|f|}{\lambda}.
$$
We finally get:
$$
|\{M_\calR f>\lambda\}|\leq (10+2^{n+1}\log^{n-2}4) \int_{\R^n} \frac{|f|}{\lambda} \left(1+\log_+^{n-2}\frac{|f|}{\lambda}\right),
$$
and the proof is complete.
\end{demo}

Using the results in section~\ref{sec.4}, it is easy to provide an example in $\R^n$, $n\geq 3$ satisfying the hypotheses of the previous proposition, for which the $L\log^{n-2}L$ estimate is sharp.
\begin{Example}\label{ex.genn} Assume $n\geq 3$ and denote by $\calR$ the family of rectangles in $\R^{n-1}$ defined in Lemma~\ref{lem.stokn2}, with $n$ replaced by $n-1$. For each $k\in\N$, denote by $\Theta_k$ and $Y_k$ the subsets of $\R^{n-1}$ associated to $\calR$ as in Lemma~\ref{lem.stokn2}. 

Now define:
$$
\tilde{\calR}:=\{R\times [0,1]: R\in\calR\},
$$
let $\tilde{Y}_k:=Y_k\times [0,1]$ and $\tilde{\Theta}_k:=\Theta_k\times [0,1]\subseteq\tilde{Y}_k$.
It is clear that one has $|\tilde{Y}_k|\geq c(n) 2^{(n-2)k} k^{n-2}|\tilde{\Theta}_k|$ with $c(n):=\frac{1}{3\cdot 2^{n-3}}$.

Observe, finally, that if $x=(x_1,\dots, x_{n-1},x_n)\in \tilde{Y}_k$ is given, one can find, according to the proof of Lemma~\ref{lem.stokn2}, a rectangle $R\in\calR$ in $\R^{n-1}$ such that we have $(x_1,\dots, x_{n-1})\in R$ and:
$$
\frac{|R\cap \Theta_k|}{|R|}\geq 2^{(2-n)k}.
$$
Letting $\tilde{R}:=R\times [0,1]$, we get $x\in\tilde{R}$ and:
$$
M_{\tilde{\calR}}\chi_{\tilde{\Theta}_k}(x)\geq\frac{ |(R\times [0,1])\cap (\Theta\times [0,1])|}{|R\times [0,1]|}\geq 2^{(2-n)k}.
$$
It hence follows that $\tilde{R}$ satisfies the hypotheses of Proposition~\ref{prop.stokn} with $d=n-2$. According to this proposition, we see that the $L\log^{n-2}L$ weak estimate for $M_{\tilde{\calR}}$ is sharp. On the other hand, it is clear that the projections of rectangles in $\tilde{\calR}$ onto the $x_{n-1}x_n$ coordinate plane form a family of rectangles in $\R^2$ having finite width.
\end{Example}

The following property, introduced by Stokolos \cite{STOKOLOS2005}, also generalizes in a straightforward way to the $n$-dimensional case.
\begin{Definition} A family of standard, dyadic rectangles in $\R^n$ satisfies property (C) if there exists a projection $p$ onto one coordinate plane enjoying the following property:
\begin{itemize}
\item[(C)] there exists an integer $k\in\N^*$ such that for any finite family of rectangles $R_{1}, \ldots, R_{k} \in {{\calR}}$ whose projections $p(R_{i})$ are \textit{ pairwise comparable}, one can find integers $1\leq i< j\leq k$ with $R_{i} \sim R_{j}$.
\end{itemize}
\end{Definition}

\begin{Proposition} \label{Prop2}
If a family of rectangles  in $\R^n$, $n \geq 3$, satisfies property (C), then it has weak type $L \log^{n-2} L$. 
\end{Proposition}

\begin{demo} Without loss of generality we can assume that $p:\R^n\to\R^2$, $(x_1,\dots, x_n)\mapsto (x_{n-1},x_n)$ 
is the projection onto the $x_{n-1}x_n$-plane. Let us prove the result by recurrence on the dimension $n$. For $n=3$, 
the conclusion follows from Theorem 3 in \cite{STOKOLOS2005}. So fix now $n\geq 4$ and assume that the result
holds in dimension $n-1$. Denote by $p_1:\R^n\to\R, (x_1\dots, x_n)\mapsto x_1$ the projection on the first axis and 
by $p_1':\R^n\to\R^{n-1}, (x_1,\dots, x_n)\mapsto (x_2,\dots, x_n)$ the projection on the last $n-1$ coordinates, and observe that, by hypothesis, the families:
$$
\calR_1:=\{p_1(R):R\in\calR\}\quad\text{and}\quad \calR_2:=\{p_1'(R):R\in\calR\}
$$
satisfy weak inequalities as in de Guzm\'an \cite[p.~185]{Guzman1974} with $\varphi_1(t):=t$ and $\varphi(t):=t(1+\log_+^{n-3}t)$. 
We proceed as in Proposition \ref{Prop1}. 
\end{demo}

\begin{Remark} When $n=3$, the above proposition is shown in \cite{STOKOLOS2005} to be sharp.
Example~\ref{ex.genn} shows again that this estimate cannot be improved in general.
\end{Remark}

\subsection*{Acknowledgements.} E. D'Aniello would like to thank the Universit\'e Paris-Sud and especially Mrs.~Martine Cordasso for their logistic and financial support in March 2016. Both authors would like to thank the Institut Henri Poincaré, and especially Mrs.~Florence Lajoinie, for their kind hospitality during the period when this work was finalized. {Last but not least, the authors would like to thank the referee for his/her careful reading and his/her nice suggestions.}

\vspace{0.3cm}
\noindent
\small{\textsc{Emma D'Aniello},}
\small{\textsc{Dipartimento di Matematica e Fisica},}
\small{\textsc{Scuola Politecnica e delle Scienze di Base},}
\small{\textsc{Seconda Universit\`a degli Studi di Napoli},}
\small{\textsc{Viale Lincoln n. 5, 81100 Caserta,}}
\small{\textsc{Italia},}
\footnotesize{\texttt{emma.daniello@unina2.it}.}

\vspace{0.3cm}
\noindent
\small{\textsc{Laurent Moonens},}
\small{\textsc{Laboratoire de Math\'ematiques d'Orsay, Universit\'e Paris-Sud, CNRS UMR8628, Universit\'e Paris-Saclay},}
\small{\textsc{B\^atiment 425},}
\small{\textsc{F-91405 Orsay Cedex},}
\small{\textsc{Fran\-ce},}
\footnotesize{\texttt{laurent.moonens@math.u-psud.fr}.}

\bibliography{biblio}

\begin{thebibliography}{10}

\bibitem{CORDOBA1979}
Antonio C{\'o}rdoba.
\newblock Maximal functions, covering lemmas and {F}ourier multipliers.
\newblock In {\em Harmonic analysis in {E}uclidean spaces ({P}roc. {S}ympos.
  {P}ure {M}ath., {W}illiams {C}oll., {W}illiamstown, {M}ass., 1978), {P}art
  1}, Proc. Sympos. Pure Math., XXXV, Part, pages 29--50. Amer. Math. Soc.,
  Providence, R.I., 1979.

\bibitem{Guzman1974}
Miguel de~Guzm{\'a}n.
\newblock An inequality for the {H}ardy-{L}ittlewood maximal operator with
  respect to a product of differentiation bases.
\newblock {\em Studia Math.}, 49:185--194, 1973/74.

\bibitem{Guzman1975}
Miguel de~Guzm{\'a}n.
\newblock {\em Differentiation of integrals in {$R^{n}$}}.
\newblock Lecture Notes in Mathematics, Vol. 481. Springer-Verlag, Berlin-New
  York, 1975.
\newblock With appendices by Antonio C{\'o}rdoba, and Robert Fefferman, and two
  by Roberto Moriy{\'o}n.

\bibitem{DILWORTH}
Robert~P. Dilworth.
\newblock A decomposition theorem for partially ordered sets.
\newblock {\em Ann. of Math. (2)}, 51:161--166, 1950.

\bibitem{FAVA1972}
Norberto~Angel Fava.
\newblock Weak type inequalities for product operators.
\newblock {\em Studia Math.}, 42:271--288, 1972.

\bibitem{FP2005}
Robert Fefferman and Jill Pipher.
\newblock A covering lemma for rectangles in {${\Bbb R}^n$}.
\newblock {\em Proc. Amer. Math. Soc.}, 133(11):3235--3241 (electronic), 2005.

\bibitem{GARSIA}
Adriano~M. Garsia.
\newblock {\em Topics in almost everywhere convergence}, volume~4 of {\em
  Lectures in Advanced Mathematics}.
\newblock Markham Publishing Co., Chicago, Ill., 1970.

\bibitem{JMZ}
B{ø}rge {Jessen}, J{\'o}zef {Marcinkiewicz}, and Antoni {Zygmund}.
\newblock {Note on the differentiability of multiple integrals.}
\newblock {\em {Fundam. Math.}}, 25:217--234, 1935.

\bibitem{MR2012}
Laurent Moonens and Joseph~M. Rosenblatt.
\newblock Moving averages in the plane.
\newblock {\em Illinois J. Math.}, 56(3):759--793, 2012.

\bibitem{SORIA1986}
Fernando Soria.
\newblock Examples and counterexamples to a conjecture in the theory of
  differentiation of integrals.
\newblock {\em Ann. of Math. (2)}, 123(1):1--9, 1986.

\bibitem{STOKOLOS1988}
Alexander~M. Stokolos.
\newblock On the differentiation of integrals of functions from {$L\varphi
  (L)$}.
\newblock {\em Studia Math.}, 88(2):103--120, 1988.

\bibitem{STOKOLOS2005}
Alexander~M. Stokolos.
\newblock Zygmund's program: some partial solutions.
\newblock {\em Ann. Inst. Fourier (Grenoble)}, 55(5):1439--1453, 2005.

\bibitem{STOKOLOS2008}
Alexander~M. Stokolos.
\newblock Properties of the maximal operators associated with bases of
  rectangles in {$\Bbb R^3$}.
\newblock {\em Proc. Edinb. Math. Soc. (2)}, 51(2):489--494, 2008.

\end{thebibliography}
\bibliographystyle{plain}
\end{document}